\documentclass[11pt,reqno]{amsart}

\usepackage{amsmath, amsthm, amssymb}
\usepackage{amsfonts}

\usepackage{hyperref}

\usepackage[ansinew]{inputenc}
\usepackage[dvips]{epsfig}
\usepackage{graphicx}
\usepackage[english]{babel}

\usepackage{thmtools}
\theoremstyle{plain}
\declaretheorem[title=Theorem, parent=section]{theorem}
\declaretheorem[title=Lemma,sibling=theorem]{lemma}

\theoremstyle{definition}

\declaretheorem[title=Remark, numbered=no]{remark*}

\declaretheorem[title=Assumption, numbered=no]{assumption*}

\numberwithin{equation}{section}

\usepackage[backgroundcolor=white, bordercolor=blue,
linecolor=blue]{todonotes}

\usepackage{dsfont}
\usepackage{bbm}

\newcommand{\N}{\mathbb{N}}
\newcommand{\R}{\mathbb{R}}

\newcommand{\cB}{\mathcal{B}}

\newcommand{\cA}{\mathcal{A}}

\newcommand{\cQ}{\mathcal{Q}}

\newcommand{\cE}{\mathcal{E}}
\newcommand{\cK}{\mathcal{K}}

\newcommand{\cG}{\mathcal{G}}

\newcommand{\eps}{\varepsilon}

\newcommand{\1}{\mathbbm{1}}

\renewcommand{\d}{\textnormal{\,d}}

\newcommand{\average}{{\mathchoice {\kern1ex\vcenter{\hrule height.4pt
width 6pt depth0pt} \kern-9.7pt} {\kern1ex\vcenter{\hrule
height.4pt width 4.3pt depth0pt} \kern-7pt} {} {} }}
\newcommand{\dashint}{\average\int}

\begin{document}
\allowdisplaybreaks
\title{The Harnack inequality fails for nonlocal kinetic equations}
 
\author{Moritz Kassmann}
\author{Marvin Weidner}

\address{Fakult\"{a}t f\"{u}r Mathematik\\Universit\"{a}t Bielefeld\\Postfach 100131\\D-33501 Bielefeld}
\email{moritz.kassmann@uni-bielefeld.de}
\urladdr{www.math.uni-bielefeld.de/$\sim$kassmann}

\address{Departament de Matem\`atiques i Inform\`atica, Universitat de Barcelona, Gran Via de les Corts Catalanes 585, 08007 Barcelona, Spain}
\email{mweidner@ub.edu}
\urladdr{https://sites.google.com/view/marvinweidner/}

\thanks{Moritz Kassmann gratefully acknowledges financial support by the German Research Foundation (SFB 1283 - 317210226). Marvin Weidner has received funding from the European Research Council (ERC) under the Grant Agreement No 801867 and the Grant Agreement No 101123223 (SSNSD), and by the AEI project PID2021-125021NA-I00 (Spain).}



\allowdisplaybreaks

\begin{abstract}
We prove that the Harnack inequality fails for nonlocal kinetic equations. Such equations arise as linearized models for the Boltzmann equation without cutoff and are of hypoelliptic type. We provide a counterexample for the simplest equation in this theory, the fractional Kolmogorov equation. Our result reflects a purely nonlocal phenomenon since the Harnack inequality holds true for local kinetic equations like the Kolmogorov equation.
\end{abstract}

\allowdisplaybreaks

\maketitle
\section{Introduction}  

The goal of this article is to show that the \textit{Harnack inequality fails} for the following nonlocal kinetic equation
\begin{align}
\label{eq:fractional-Kolmogorov}
\partial_t f + v \cdot \nabla_x f + (-\Delta_v)^s f = 0 ~~ \text{ in } \cQ \subset \R \times \R^{d} \times \R^d
\end{align}
and its time-independent version
\begin{align}
\label{eq:fractional-hypoelliptic}
v \cdot \nabla_x f + (-\Delta_v)^s f &= 0 ~~ \text{ in } \cB \subset \R^{d} \times \R^d,
\end{align}
where $s \in (0,1)$, $d \in \N$, and
\begin{align*}
(-\Delta_v)^s f(x,v) = c_{d,s} ~\text{p.v.}~\int_{\R^d} \big( f(x,v) - f(x,w) \big) |v-w|^{-d-2s} \d w
\end{align*}
denotes the fractional Laplacian acting only on the $v$-variable. Here $c_{d,s} > 0$ denotes a normalization constant. The equation \eqref{eq:fractional-hypoelliptic} is the kinetic fractional Kolmogorov equation, which arises as a linearized model for the Boltzmann equation without Grad's cutoff assumption.

As we will explain below, our result reflects a purely \textit{kinetic nonlocal phenomenon}, since the Harnack inequality holds true for kinetic local equations (i.e., when $s = 1$ in \eqref{eq:fractional-hypoelliptic}) \cite{LaPo94,Mou18,GIMV19,GuMo22}, and also for non-degenerate nonlocal equations \cite{BaLe02,ChKu03, DKP14,Coz17,KaWe23}.

\begin{theorem}
\label{thm:counterexample}
Let $s \in (0,1)$. There exist a constant $c_0 > 0$ and solutions $f_{\eps} : \R^d \times \R^d \to [0,1]$ to
\begin{align*}
v \cdot \nabla_x f_{\eps}(x,v) + (- \Delta_v)^s f_{\eps}(x,v) = 0 ~~ \text{ for } (x,v) \in \cB:= B_1(0) \times B_1(0),
\end{align*}
such that for $\zeta = (\frac{1}{2} e_d,0) \in \R^d \times \R^d$ and every $\eps \in (0,\frac{1}{4})$ it holds $f_{\eps} \not\equiv 0$ and
\begin{align*}
f_{\eps}(\zeta) \le c_0 \, \eps^{d(1+2s) -2s} f_{\eps}(0).
\end{align*}
In particular, the ratio $f_{\eps}(0)/f_{\eps}(\zeta) \to \infty$ as $\eps \to 0$.
\end{theorem}

Thus, \autoref{thm:counterexample} implies the failure of the Harnack inequality for \eqref{eq:fractional-hypoelliptic}, and in particular for \eqref{eq:fractional-Kolmogorov}, because $f_{\eps}$ is independent of $t$.

In the following, we first comment on the relation between \autoref{thm:counterexample} and the recently developed theories for nonlocal hypoelliptic and kinetic equations. Second, we consider its relation to existing results on Harnack inequalities for non-degenerate nonlocal problems. Finally, we explain the strategy of the proof and the definition of the functions $f_{\eps}$.

\subsection{Background on nonlocal kinetic equations}

The fractional Kolmogorov equation
\begin{align*}
\partial_t f + v \cdot \nabla_x f + (-\Delta_v)^s f = 0 ~~ \text{ in } \cQ
\end{align*}
can be regarded as a linearized model for the Boltzmann equation without cutoff. The regularity program for the Boltzmann equation subject to macroscopic bounds has recently been established by Imbert, Silvestre, and Mouhot in a series of works \cite{Sil16,IMS20,ImSi20,ImSi21,ImSi22} (see also the survey \cite{ImSi20b}). A central part of this theory is the development of the regularity theory for solutions to the following more general class of nonlocal kinetic equations
\begin{align}
\label{eq:kinetic-PDE}
\partial_t f + v \cdot \nabla_x f + L_v f = 0 ~~ \text{ in } \cQ,
\end{align}
where $\cQ$ is a kinetic cylinder and $L_v$ denotes an integro-differential operator, acting only on $v \in \R^d$, defined as
\begin{align*}
L_v f(v) = \text{p.v.}~\int_{\R^d} \big( f(v) - f(w) \big) K_{t,x}(v,w) \d w,
\end{align*}
where, for any $t,x$, the kernel $K_{t,x} : \R^d \times \R^d \to [0,\infty]$ satisfies suitable ellipticity and boundedness assumptions. One key step in their program (see \cite{ImSi20}) consists in proving a \textit{weak Harnack inequality} for nonnegative (super)solutions to \eqref{eq:kinetic-PDE} without any regularity assumptions on the kernel, i.e., for some $\eps_0 \in (0,1)$, $C > 0$ independent of $f$, and suitable kinetic cylinders $\cQ^{\ominus}, \cQ^{\oplus}$ the following estimate
\begin{align}
\label{eq:weak-Harnack}
\left( \dashint_{\cQ^{\ominus}} f^{\eps_0} \right)^{1/\eps_0} \le C \inf_{\cQ^{\oplus}} f.
\end{align}
The weak Harnack inequality is used to deduce an a priori H\"older regularity estimate for solutions to \eqref{eq:kinetic-PDE} and also for the Boltzmann equation. This result can be regarded as a De Giorgi-Nash-Moser type theorem for nonlocal kinetic equations (see also \cite{Sto19}). 

The question whether even the strong Harnack inequality holds true for \eqref{eq:kinetic-PDE}, i.e., whether there exists $C > 0$, depending only on $d,s,$ and the ellipticity constants of $K_{t,x}$, such that
\begin{align}
\label{eq:strong-Harnack}
\sup_{\cQ^{\ominus}} f \le C \inf_{\cQ^{\oplus}} f
\end{align}
has been an open problem after \cite{Sil16,IMS20,ImSi20,ImSi21,ImSi22}. Our main result \autoref{thm:counterexample} provides a negative answer to this question, already in case $L_v = (-\Delta_v)^s$ and for any dimension $d \in \N$. It even shows that the Harnack inequality fails for solutions that are stationary in $t$.
Moreover, due to \eqref{eq:weak-Harnack}, \autoref{thm:counterexample} implies that nonnegative weak solutions to \eqref{eq:kinetic-PDE} do not satisfy a local boundedness estimate only in terms of local quantities
\begin{align}
\label{eq:loc-bd}
\sup_{\cQ^{\ominus}} f \le C \left( \dashint_{2\cQ^{\ominus}} f^2 \right)^{1/2}.
\end{align}
Moreover, a close inspection of our counterexample \eqref{eq:counterexample} shows that \eqref{eq:loc-bd} remains false if a tail term (see \cite{KaWe23}) is added to the right-hand side in \eqref{eq:loc-bd}, as long as the tail is $L^p$ in $x$ for some $p < \frac{d(1+2s)}{2s}$.
The failure of both, the Harnack inequality \eqref{eq:strong-Harnack}, and the local boundedness \eqref{eq:loc-bd} of weak solutions to \eqref{eq:kinetic-PDE} (and to \eqref{eq:fractional-hypoelliptic}), is in stark contrast to the theory of nonlocal elliptic and parabolic equations (see \cite{DKP14,DKP16,Coz17}, and \cite{KaWe23}).

%
%
%

\cite{Loh24b} establishes a nonlinear version of the Harnack inequality \eqref{eq:strong-Harnack} for weak solutions to \eqref{eq:kinetic-PDE}. This result is interesting in itself and does not contradict \autoref{thm:counterexample}. Our counterexample implies that the exponent $\beta \in (0,1)$ in \cite[Theorem 1.3]{Loh24b} is at most $\frac{2s}{d(1+2s)}$. 

\autoref{thm:counterexample} has recently stimulated further research on variants of the Harnack inequality \eqref{eq:strong-Harnack} and the local boundedness estimate \eqref{eq:loc-bd} for nonlocal kinetic equations \eqref{eq:kinetic-PDE}.
\cite{Loh24a} states a Harnack inequality for global solutions under additional assumptions, see \cite[Definition 2.2 (iii), (iv)]{Loh24a}. Moreover, \cite{APP24} establishes \eqref{eq:loc-bd} with a tail term in $L^p$ for any $p > \frac{4s + 4d + 4ds}{2s}$.

From the perspective of local kinetic equations, the failure of the Harnack inequality for nonlocal equations such as \eqref{eq:fractional-Kolmogorov} and \eqref{eq:fractional-hypoelliptic} comes as a surprise. Indeed, in \cite{LaPo94}, \cite{KoPo16}, \cite{AbTr19}, \cite{AEM19}, \cite{PaPo04},  Harnack inequalities have been established for different classes of \textit{local} hypoelliptic equations. Moreover, in \cite{Mou18,GIMV19,GuMo22} the authors prove Harnack inequalities for local kinetic equations driven by a second order operator in divergence form acting on $v$, using approaches via Moser and De Giorgi iteration, respectively. Hence, \autoref{thm:counterexample} describes a purely nonlocal effect, on which we will elaborate in the discussion below.

Note that in the past years there has been an increasing interest in the investigation of nonlocal hypoelliptic operators of H\"ormander type
\begin{align*}
\mathcal{K}f := v \cdot \nabla_x f + L_v f,
\end{align*}
and there are many parallels between the local and nonlocal theories of hypoellipticity. For instance, we refer to \cite{AIn24a}, \cite{AIN24b} for the construction of fundamental solutions to the corresponding nonlocal kinetic Cauchy problems, to \cite{ChZh18}, \cite{HMP19}, \cite{NiZa21}, \cite{Nie22}, \cite{NiZa22} for nonlocal kinetic $L^p$ maximal regularity results, to \cite{Loh23}, \cite{ImSi21}, \cite{HWZ20} for nonlocal kinetic Schauder theory, and to \cite{HuMe16}, \cite{ChZh18}, \cite{HWZ20}, \cite{HPZ21} for approaches to nonlocal kinetic and hypoelliptic equations via degenerate stochastic differential equations driven by L\'evy processes.

\subsection{Background on Harnack inequalities for nonlocal operators}

The development of the De Giorgi-Nash-Moser theory for nonlocal equations, and in particular the investigation of Harnack inequalities has become an active field of research in the past 20 years. In the simplest case
\begin{align}
\label{eq:s-harm}
(-\Delta)^s f = 0 ~~ \text{ in } B_1,
\end{align}
the Harnack inequality  for globally nonnegative solutions follows from classical works, including \cite{Rie38}, \cite{Lan71}. It was shown in \cite{Kas07} that the Harnack inequality fails for \eqref{eq:s-harm} unless one assumes that $u \ge 0$ in $\R^d$. If solutions are nonnegative only in $B_1$, one can still establish a nonlocal version of the Harnack inequality including a so-called tail term, which captures the effect caused by the nonlocality of the operator (see \cite{Kas11,DKP16,DKP14,Coz17}). Since \autoref{thm:counterexample} describes the failure of the Harnack inequality for globally nonnegative solutions, we will only focus on this case in the sequel.

A challenging topic is the investigation of Harnack inequalities for elliptic and parabolic equations
\begin{align}
\label{eq:L-harm}
L f = 0 ~~ \text{ in } B_1, \qquad \partial_t f + L f = 0 ~~ \text{ in } Q_1
\end{align}
driven by more general nonlocal operators of the form 
\begin{align*}
L f(v) = \text{p.v.}~\int_{\R^d} \big( f(v) - f(w) \big) \mu(v,\d w),
\end{align*}
where $(\mu(v,\cdot))_{v\in \R^d}$ is a family of measures on $\R^d$, satisfying suitable assumptions. We refer to \cite{BaLe02}, \cite{ChKu03}, \cite{BaKa05}, \cite{CaSi09}, \cite{DKP14}, \cite{Coz17}, \cite{Str19}, \cite{CKW23} and to \cite{ChDa16}, \cite{CKW19}, \cite{CKW20}, \cite{KaWe24} for proofs of Harnack inequalities under fairly general assumptions on $\mu$. Moreover, only recently, Harnack inequalities have been established for a general class of parabolic nonlocal equations in \cite{KaWe23}, using an entirely analytic approach.

It is important to point out that the Harnack inequality holds true in the aforementioned works only under suitable assumptions on $\mu$ that go beyond classical ellipticity conditions. In fact, there exist many classes of uniformly elliptic examples $\mu$ for which the Harnack inequality fails (see \cite{BoSz05}, \cite{BaKa05}, \cite{BaCh10}, \cite{MaMu22}, \cite{Kit23}). Most of these counterexamples have in common that $\mu$ exhibits \textit{long range oscillations}, namely that the value of $\mu(v,A)$ can vary drastically for a given set $A \subset \R^d \setminus B_2$, even if $v \in B_1$ only changes marginally. A simple example is given by
\begin{align*}
\mu_{\text{axes}}(v,\d w) := \sum_{i = 1}^d |v_i - w_i|^{-1-2s} \d w_i \prod_{j \not = i}^d \delta_{\{v_j\}}(\d w_j).
\end{align*}  
The measure $\mu_{\text{axes}}(v,\d w)$ charges only those sets that intersect the coordinate axes centered at $v$, which leads to long range oscillations, as described before. As a consequence, a solution to \eqref{eq:L-harm} might oscillate around a point $v \in B_1$, depending on whether $\mu_{\text{axes}}(v,\d w)$ interacts with certain portions of the exterior of the solution domain for suitable exterior data. 

As we will see in the next section, the main observation in the proof of \autoref{thm:counterexample} is that, although $(-\Delta_v)^s$ is a well-behaved, non-degenerate diffusion operator in $v$, the combination with the drift term $v \cdot \nabla_x$ makes the resulting operator $\mathcal{K} = v \cdot \nabla_x + (-\Delta_v)^s$  sensitive to changes of the exterior data when the $x$-variable is varied. This anisotropic behavior leads to an effect that is very similar to the one present for $\mu_{\text{axes}}$, and becomes most apparent in our proof of \autoref{thm:counterexample} through \eqref{eq:formula-u-eps}. 

Our result shows that the aforementioned anisotropy of $\mathcal{K}$ leads to a failure of the Harnack inequality. This phenomenon is quite remarkable since the degeneracy of $\cK$ is no obstruction to \textit{$C^{\infty}$ regularity of bounded solutions} to the fractional Kolmogorov equation \eqref{eq:fractional-Kolmogorov} (see \cite{ImSi20,ImSi21,ImSi22}). In fact, velocity averaging lemmas are available for $\mathcal{K}$ (see \cite{ImSi20}, \cite{Loh23}, \cite{NiZa21}) and lead to a regularity transfer from the $v$-variable to the $x$-variable in complete analogy to the situation for local  operators falling within the realm of H\"ormander's hypoellipticity theory.

Let us close the discussion with a remark on why the new techniques for nonlocal parabolic equations developed in \cite{KaWe23} do not apply to nonlocal kinetic equations. A main idea in \cite{KaWe23} was to estimate the parabolic tail (which is $L^1$ in the $t$-variable) by local quantities and to combine this result with a local boundedness estimate involving the $L^1$ tail. Note that one cannot expect to bound the $L^p$ tail for $p > 1$ by local quantities, since weak solutions a priori do not need to have finite $L^p$ tails. On the other hand, the local boundedness estimate involving the $L^1$ tail is a borderline result, since the tail can be considered as a source term belonging to $L^{1,\infty}_{t,x}$, which is critical with respect to parabolic scaling. For kinetic equations, due to the same reason, there is no hope to estimate any other tail than the $L^{1,1}_{t,x}$ tail from above by a local quantity. However,  a local boundedness estimate involving the $L^{1,1}_{t,x}$ tail fails, since such tails would correspond to source terms in $L^{1,1,\infty}_{t,x,v}$, which is supercritical with respect to kinetic scaling (see also the discussion after \eqref{eq:loc-bd}). We refer to \cite{Sto19} for an extension of the results in \cite{ImSi20} including unbounded source terms.

\subsection{Strategy of proof}

Previous proofs of the failure of the Harnack inequality for non-degenerate nonlocal equations \eqref{eq:L-harm} rely on purely probabilistic arguments (see \cite{BoSz05}, \cite{BaCh10}, \cite{MaMu22}), involving exit- and hitting-time estimates for the corresponding Markov jump processes. In contrast to that, our proof of \autoref{thm:counterexample} is of analytic nature and uses only basic estimates of the fundamental solution, scaling properties, and the weak maximum principle (see Section \ref{sec:prelim}). 

In the sequel, let us define the functions $f_{\eps}$ and sketch the main arguments of our proof of \autoref{thm:counterexample}. We define the following subsets of $\R^d \times \R^d$ for $\eps \in (0,1)$
\begin{align*}
\cB = B_1(0) \times B_1(0), \qquad \cG_{\eps} = B_{\eps^{1+2s}}(0) \times B_1(0), \qquad \cE_{\eps} = B_{\eps^{1+2s}}(0) \times B_1(3e_d)
\end{align*}
and let $f_{\eps}$ be the solution to 
\begin{align}
\label{eq:counterexample}
\left\{\begin{array}{rcl}
v \cdot \nabla_x f_{\eps} + (-\Delta_v)^s f_{\eps} &=& 0 ~~~~~ \text{ in } \cB,\\
f_{\eps} &=& \1_{\cE_{\eps}} ~~ \text{ in } (\R^d \times \R^d) \setminus \cB. 
\end{array}\right.
\end{align}

The proof of \autoref{thm:counterexample} relies on the following two lemmas, whose proofs are given in Section \ref{sec:lemma1} and Section \ref{sec:lemma2}, respectively.

\begin{lemma}
\label{lemma:u-lower-bound}
There exists $c_1 > 0$ such that for any $\eps \in (0,1)$:
\begin{align*}
f_{\eps}(0) \ge c_1 \eps^{2s}.
\end{align*}
\end{lemma}

\begin{lemma}
\label{lemma:u-upper-bound}
Let $\zeta := (\frac{1}{2} e_d,0) \in \R^d \times \R^d$. There exists $c_2 > 0$ such that for any $\eps \in (0,\frac{1}{4})$:
\begin{align*}
f_{\eps}(\zeta) \le c_2 \eps^{d(1+2s)}.
\end{align*}
\end{lemma}

Clearly, \autoref{lemma:u-lower-bound} and \autoref{lemma:u-upper-bound} immediately imply our main result \autoref{thm:counterexample}.

\begin{proof}[Proof of \autoref{thm:counterexample}]
By \autoref{lemma:u-lower-bound} and \autoref{lemma:u-upper-bound}, we have
\begin{align*}
f_{\eps}(\zeta) \le c_2 \eps^{d(1+2s)} \le c_2 c_1^{-1} \eps^{d(1+2s)-2s} f_{\eps}(0).
\end{align*}
Thus, we conclude the proof by setting $c_0 := c_2 c_1^{-1}$.
\end{proof}

\subsection{Outline}
This article is structured as follows. In Section \ref{sec:prelim}, we collect several auxiliary results, introduce the fundamental solution to the fractional Kolmogorov equation and derive a helpful identity for $f_{\eps}$ in \eqref{eq:formula-u-eps}. Then, in Section \ref{sec:lemma1} and Section \ref{sec:lemma2} we establish \autoref{lemma:u-lower-bound} and \autoref{lemma:u-upper-bound}, respectively.

\subsection*{Acknowledgments}

We thank Florian Grube for helpful comments on the manuscript. Moreover, we are grateful to Florian Grube and Tuhin Ghosh for sharing their insights regarding \eqref{eq:integral-in-v}.

\subsection*{Data availability statement}

Data sharing is not applicable to this article as no data sets were generated or analyzed.

\section{Preliminaries}
\label{sec:prelim}

Let us collect several key properties that will be important for our approach.

Our proof of \autoref{thm:counterexample} makes use of the fundamental solution $(t,(x,v),(y,w)) \mapsto p_t(x,v;y,w)$ to the fractional Kolmogorov equation \eqref{eq:fractional-Kolmogorov} in $(0,\infty) \times \R^d \times \R^d$, i.e., for any $g \in L^{2}(\R^d \times \R^d)$ it holds that $f : \R^d \times \R^d \to \R$ given by
\begin{align*}
f(x,v) = \int_0^{\infty} \int_{\R^d} \int_{\R^d} p_t(x,v;y,w)g(y,w) \d y \d w \d t
\end{align*}
is the solution to
\begin{align*}
v \cdot \nabla_x f + (-\Delta_v)^s f = g ~~ \text{ in } \R^d \times \R^d.
\end{align*}
The fundamental solution is smooth, nonnegative, and it satisfies the relation
\begin{align*}
p_t(x,v;y,w) = P_t(x-y-tw,v-w),
\end{align*}
where the function $(t,x,v) \mapsto P_t(x,v)$ satisfies the following scaling property for any $\eps > 0$
\begin{align}
\label{eq:scaling}
P_{t/\eps^{2s}}(x,v) = \eps^{d(2+ 2s)} P_t(\eps^{1+2s} x, \eps v),
\end{align}
and $P_t$ is defined as follows, via the  Fourier transform:
\begin{align*}
\mathcal{F}[P_t](\eta,\xi) = \exp \left( - \int_0^t |\xi + \tau \eta|^{2s} \d \tau \right).
\end{align*}
We refer to \cite[Section 2.3, 2.4]{ImSi20} as a reference to all of the aforementioned properties. 
Moreover, we have the following transformation invariance $P_t(x,v) = P_t(tv-x,v) = P_t(x-tv,-v)$, which follows immediately from the definition.

Also in our work, the asymptotic behavior of the fundamental solution plays an important role. We require the following two elementary properties of $P_t(x,v)$.

\begin{lemma}
\label{lemma:fund-sol-bounds}
The function $P_t(x,v)$ satisfies for some constant $c > 0$
\begin{align*}
\int_{\R^d} P_t(x,w) \d w \le c t^{-d-\frac{d}{2s}} \big(1 + t^{-1-\frac{1}{2s}} |x| \big)^{-d-2s} ~~ \forall t > 0,~ \forall x \in \R^d.
\end{align*}
\end{lemma}

\begin{proof}
We observe that for any $x \in \R^d$
\begin{align}
\label{eq:integral-in-v}
\int_{\R^d} P_1(x,w) \d w \le c (1 + |x|)^{-d-2s}.
\end{align}
This property follows immediately from the computation
\begin{align*}
\int_{\R^d} P_1(x,w) \d w &= \mathcal{F}^{-1}_{\eta} \left[ \int_{\R^d} \mathcal{F}^{-1}_{\xi} \left[ e^{- \int_0^1 |\xi + \tau \eta|^{2s} \d \tau} \right](w) \d w \right](x) \\
& = \mathcal{F}^{-1}_{\eta} \left[ e^{ - (1+2s)^{-1}|\eta|^{2s}} \right](x) = (1+2s)^{\frac{d}{2s}} Q_1 \big((1+2s)^{\frac{1}{2s}} x \big),
\end{align*}
where we denote $Q_1(x) := \mathcal{F}^{-1}_{\eta} \left[ e^{-|\eta|^{2s}} \right](x)$. Note that $Q_t(x) := t^{-\frac{d}{2s}}Q_1(t^{-\frac{1}{2s}}x)$ is the fractional heat kernel, i.e., the fundamental solution to $\partial_t Q_t + (-\Delta)^s Q_t = 0$ in $(0,\infty) \times \R^d$, which satisfies $Q_1(x) \le c(1+|x|)^{-d-2s}$. Hence, we deduce \eqref{eq:integral-in-v}. Then, by scaling \eqref{eq:scaling} for any $x \in \R^d$
\begin{align*}
\int_{\R^d} P_t(x,w) \d w &= \int_{\R^d} t^{-d-\frac{d}{s}} P_1( t^{-1-\frac{1}{2s}} x, t^{-\frac{1}{2s}} w) \d w = t^{-d-\frac{d}{2s}}  \int_{\R^d}  P_1( t^{-1-\frac{1}{2s}} x, w) \d w \\
&\le c t^{-d-\frac{d}{2s}} (1 + t^{-1-\frac{1}{2s}} |x|)^{-d-2s}, 
\end{align*}
as desired.
\end{proof}

Next, given any bounded, open set $\Omega \subset \R^d \times \R^d$, we introduce the fundamental solution $(t,(x,v),(y,w)) \mapsto p_t^{\Omega}(x,v;y,w)$ to the 
Dirichlet problem associated to \eqref{eq:fractional-Kolmogorov} in $\Omega$, for which it holds that for any $g \in L^{2}(\Omega)$ the function
\begin{align*}
f(x,v) = \int_0^{\infty}\iint_{\Omega} p_t^{\Omega}(x,v;y,w)g(y,w) \d y \d w \d t
\end{align*}
is the solution to
\begin{align}
\label{eq:Dirichlet-formula}
\left\{\begin{array}{rcl}
v \cdot \nabla_x f + (-\Delta_v)^s f &=& g ~~ \text{ in } \Omega,\\
f &=& 0 ~~ \text{ in } (\R^d \times \R^d) \setminus \Omega.
\end{array}\right.
\end{align}
Establishing the existence and uniqueness of the fundamental solution $p^{\Omega}_t$ follows by standard arguments. It can be achieved, for instance, by studying the well-posedness of the fractional Kolmogorov equations \eqref{eq:fractional-Kolmogorov}, but also via the theory of stochastic processes, where $p_t^{\Omega}$ is the transition density of the Markov process $Z_t = (x - \int_0^t X_s \d s , v + X_t)$ in $\R^d \times \R^d$ (which is generated by $\cK = v \cdot \nabla_x + (-\Delta_v)^s$), killed upon exiting $\Omega$, where $X_t$ denotes the rotationally symmetric $2s$-stable process in $\R^d$ (see for instance \cite{ChZh18}).

Note that for any two sets $\Omega \subset \Omega' \subset \R^d \times \R^d$ and $\cA \subset \R^d \times \R^d$, we have
\begin{align}
\label{eq:heat-kernel-comp}
0 \le \iint_{\cA} p_t^{\Omega}(x,v;y,w) \d y \d w \le \iint_{\cA} p_t^{\Omega'}(x,v;y,w) \d y \d w \qquad \forall t> 0, ~~ x,v \in \R^d.
\end{align}
This relation remains true for $\Omega'= \R^d \times \R^d$, in which case we set $p^{\R^d \times \R^d}_t := p_t$. To see this, observe that the two functions in \eqref{eq:heat-kernel-comp} solve the fractional Kolmogorov equation $\partial_t f + v \cdot \nabla_x f + (-\Delta)^s f = 0$ in $(0,\infty) \times \Omega$ and in $(0,\infty) \times \Omega'$, respectively, with $f \equiv 0$ in $(0,\infty) \times (\R^d \times \R^d) \setminus \Omega$, and in $(0,\infty) \times (\R^d \times \R^d) \setminus \Omega'$, respectively, and initial data $f(0) = \1_{\cA}$. Then, \eqref{eq:heat-kernel-comp} follows immediately from the weak maximum principle (see \cite[Lemma A.12]{ImSi20}).

Note that in the same way as for $p_t$, we have a scaling property for $p_t^{\Omega}$. Namely, let us define $p_t^{\Omega}(x,v) := p_t^{\Omega}(0,0;x,v) \ge 0$.
Then, it holds (see \cite[Section 2.3]{ImSi20})
\begin{align}
\label{eq:scaling-killed}
p_{t/\eps^{2s}}^{\cB^{\eps}}(x,v) = \eps^{d(2+ 2s)} p_t^{\cB}(\eps^{1+2s} x, \eps v), ~~ \text{ where } \cB^{\eps} := B_{\eps^{-1-2s}} \times B_{\eps^{-1}}.
\end{align}

We can use the previous insights to derive a formula for $f_{\eps}$, which will be of fundamental importance to us in the proofs of \autoref{lemma:u-lower-bound} and \autoref{lemma:u-upper-bound}. Let us define
\begin{align*}
g_{\eps}(x,v) := -(-\Delta_v)^s \1_{\cE_{\eps}}(x,v) = c_{d,s} \1_{B_{\eps^{1+2s}}(0)}(x) \int_{B_1(3e_d)} |v-w|^{-d-2s} \d w, \qquad (x,v) \in \cB,
\end{align*}
and deduce that for some $0 < C_1 \le C_2 < \infty$ it holds
\begin{align}
\label{eq:f_eps-bd}
C_1 \1_{\cG_{\eps}} \le g_{\eps} \le C_2 \1_{\cG_{\eps}} ~~ \text{ in } \cB.
\end{align}
Moreover, it holds by construction for $h_{\eps} := f_{\eps} - \1_{\cE_{\eps}}$:
\begin{align*}
\left\{\begin{array}{rcl}
v \cdot \nabla_x h_{\eps} + (-\Delta_v)^s h_{\eps} &=& g_{\eps} ~~ \text{ in } \cB, \\
h_{\eps} &=& 0 ~~~ \text{ in } (\R^d \times \R^d) \setminus \cB.
\end{array}\right.
\end{align*}
Now, since we can represent $h_{\eps}$ with the help of the fundamental solution $p_t^{\cB}$ (see \eqref{eq:Dirichlet-formula}), and using also that $f_{\eps} \equiv h_{\eps}$ in $\cB$, and \eqref{eq:f_eps-bd}, we have
\begin{align}
\label{eq:formula-u-eps}
f_{\eps}(x,v) = \int_0^{\infty} \iint_{\cG_{\eps}} p_t^{\cB}(x,v;y,w) g_{\eps}(y,w) \d y \d w \d t ~~ \forall (x,v) \in \cB.
\end{align}
Here, the key observation is that the integration in $y,w$ in \eqref{eq:formula-u-eps} takes place only over the set $\cG_{\eps}$.

\section{Proof of \autoref{lemma:u-lower-bound}}
\label{sec:lemma1}

The goal of this section is to prove \autoref{lemma:u-lower-bound}. It relies mainly on \eqref{eq:formula-u-eps} and a scaling argument.

\begin{proof}[Proof of \autoref{lemma:u-lower-bound}]
First, we deduce from \eqref{eq:formula-u-eps},  \eqref{eq:f_eps-bd}, and the nonnegativity of $p_t^{\cB}$
\begin{align*}
f_{\eps}(0) = \int_0^{\infty}\iint_{\cG_{\eps}} p_t^{\cB}(y,w) g_{\eps}(y,w) \d y \d w \d t \ge C_1 \int_0^{\infty}\int_{B_{\eps^{1+2s}}} \int_{B_{\eps}} p_t^{\cB}(y,w) \d w \d y \d t.
\end{align*}
Recall that $\cB^{\eps} := B_{\eps^{-1-2s}} \times B_{\eps^{-1}}$. Then, we deduce by scaling (see \eqref{eq:scaling-killed})
\begin{align*}
\int_0^{\infty}& \int_{B_{\eps^{1+2s}}} \int_{B_{\eps}} p_t^{\cB}(y,w) \d w \d y \d t \\
& \quad = \eps^{d(2+2s)} \int_0^{\infty}\iint_{\cB} p_t^{\cB}(\eps^{1+2s}y,\eps w) \d w \d y \d t = \int_0^{\infty}\iint_{\cB} p_{t/\eps^{2s}}^{\cB^{\eps}}(y,w) \d w \d y \d t \\
&\quad = \eps^{2s} \int_0^{\infty}\iint_{\cB} p_{t}^{\cB^{\eps}}(y,w) \d w \d y \d t \ge \eps^{2s} \int_0^{\infty}\iint_{\cB} p_{t}^{\cB}(y,w) \d w \d y \d t,
\end{align*}
where we used in the last step that, since $\cB \subset \cB^{\eps}$, by \eqref{eq:heat-kernel-comp} it holds 
\begin{align*}
\iint_{\cB} p_{t}^{\cB^{\eps}}(y,w) \d w \d y \ge \iint_{\cB} p_{t}^{\cB}(y,w) \d w \d y ~~ \forall t > 0.
\end{align*}
Finally, we observe that
\begin{align*}
\int_0^{\infty}\iint_{\cB} p_{t}^{\cB}(y,w) \d w \d y \d t > 0,
\end{align*}
which is immediate, and that this quantity is independent of $\eps$.
\end{proof}

\section{Proof of \autoref{lemma:u-upper-bound}}
\label{sec:lemma2}

The goal of this section is to prove \autoref{lemma:u-upper-bound}.
Its proof relies mainly on \eqref{eq:formula-u-eps} and \autoref{lemma:fund-sol-bounds}.

\begin{proof}[Proof of \autoref{lemma:u-upper-bound}]
First, we observe that by \eqref{eq:formula-u-eps}, \eqref{eq:f_eps-bd}, and \eqref{eq:heat-kernel-comp}, as well as the Galilean invariance $p_t(\zeta;y,w) = P_t(\frac{1}{2}e_d-y-tw,-w) = P_t(\frac{1}{2}e_d - y,w)$:
\begin{align*}
f_{\eps}(\zeta) \le C_2 \int_0^{\infty} \iint_{\cG_{\eps}} p_t(\zeta;y,w) \d y \d w \d t = C_2 \int_0^{\infty} \int_{B_{\eps^{1+2s}}} \int_{B_1} P_t \left(\frac{1}{2} e_d  - y , w \right) \d w \d y \d t.
\end{align*}
Then, by the integrated on-diagonal upper bound for $P_t$ (see \autoref{lemma:fund-sol-bounds}), we deduce
\begin{align*}
\int_{1}^{\infty} \int_{B_{\eps^{1+2s}}} \int_{B_1} P_t \left(\frac{1}{2}e_d - y, w \right) \d w \d y \d t \le c \eps^{d(1+2s)} \int_{1}^{\infty}t^{-d-\frac{d}{2s}} \d t \le c \eps^{d(1+2s)}.
\end{align*}
Next, observing that for $y \in B_{\eps^{1+2s}}$ it holds $|\frac{1}{2}e_d -y| \in (\frac{1}{4},1)$, once $\eps \in (0,\frac{1}{4})$, we obtain by using again the integrated off-diagonal upper bound for $P_t$ (see \autoref{lemma:fund-sol-bounds}):
\begin{align*}
\int_0^{1} \int_{B_{\eps^{1+2s}}} \int_{B_1} P_t \left(\frac{1}{2}e_d - y, w \right) \d w \d y \d t &\le c \eps^{d(1+2s)} \sup_{t \in (0,1)} \sup_{|x| \in (\frac{1}{4},1)} \int_{\R^d} P_t(x,w) \d w \\
&\le c \eps^{d(1+2s)}\sup_{t \in (0,1)} \sup_{|x| \in (\frac{1}{4},1)} t^{1+2s} |x|^{-d-2s} \le c \eps^{d(1+2s)}.
\end{align*}
Combining the previous two estimates yields the desired result.
\end{proof}

\end{document}